\documentclass[11pt,reqno]{amsart}
\usepackage{amssymb}
\usepackage{amsfonts}
\usepackage{mathrsfs}
\usepackage{amsmath}
\usepackage{graphicx}
\usepackage{hyperref}
\usepackage{float}
\usepackage{epstopdf}
\usepackage{color}
\usepackage{bm}
\usepackage{comment}
\usepackage{soul}

\allowdisplaybreaks
\newtheorem{theorem}{Theorem}[section]
\newtheorem{proposition}[theorem]{Proposition}
\newtheorem{lemma}[theorem]{Lemma}

\newtheorem{remark}[theorem]{Remark}

\newtheorem{definition}[theorem]{Definition}

\newenvironment{notation}{\smallskip{\sc Notation.}\rm}{\smallskip}

\def\mcE{\mathcal{E}}
\def\mcF{\mathcal{F}}

\def\mcV{\mathcal{V}}

\numberwithin{equation}{section}

\begin{document}
\title[eigenvalues of Laplacians on higher dimensional Vicsek set graphs]{eigenvalues of Laplacians on higher dimensional Vicsek set graphs}

\author{Shiping Cao}
\address{Department of Mathematics, Cornell Univeristy, Ithaca 14853, USA}
\email{sc2873@cornell.edu}
\thanks{}

\author{Robert S. Strichartz}
\address{Department of Mathematics, Cornell Univeristy, Ithaca 14853, USA}
\email{str@math.cornell.edu}
\thanks{}

\author{Melissa Wei}
\address{Department of Mathematics, Cornell Univeristy, Ithaca 14853, USA}
\email{mlw292@cornell.edu}
\thanks{}


\date{}


\begin{abstract}
	We study the graphs associated with Vicsek sets in higher dimensional settings. First, we study the eigenvalues of the Laplacians on the approximating graphs of the Vicsek sets, finding a general spectral decimation function. This is an extension of earlier results on two dimensional Vicsek sets. Second, we study the Vicsek set lattices, which are natural analogues to the Sierpinski lattices. We have a criterion when two different Vicsek set lattices are isomorphic. 
\end{abstract}
\maketitle

\section{Introduction}
	The Vicsek sets are among the simplest examples of p.c.f. self-similar sets. In some sense, we can view the Vicsek sets $\mathcal{V}_n^d$ as trees rooted at the central point $p_0$, with $2^d$ identical branches (connected components of $\mathcal{V}_n^d/\{p_0\}$), where $d$ is the dimension. See Figure \ref{fig1} for examples of Vicsek sets. As a consequence, the Vicsek set has full $S_{2^d}$ symmetry. Moreover, we can see infinitely many isometries on the Vicsek sets, which only interchange points in a small cell (see \cite{KL,HM}). 

\begin{figure}[htp]
	\includegraphics[width=4cm]{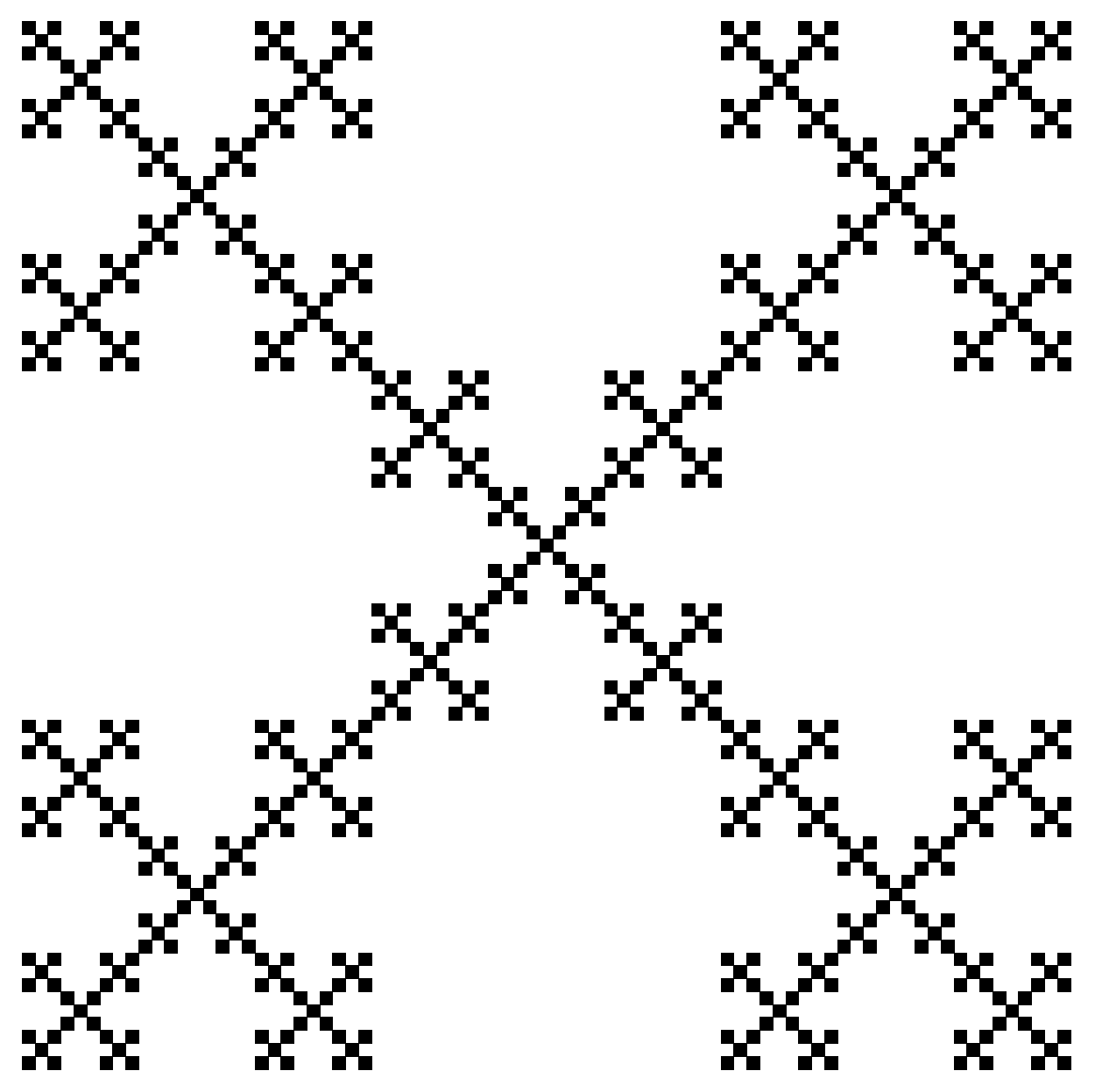}\quad
	\includegraphics[width=4cm]{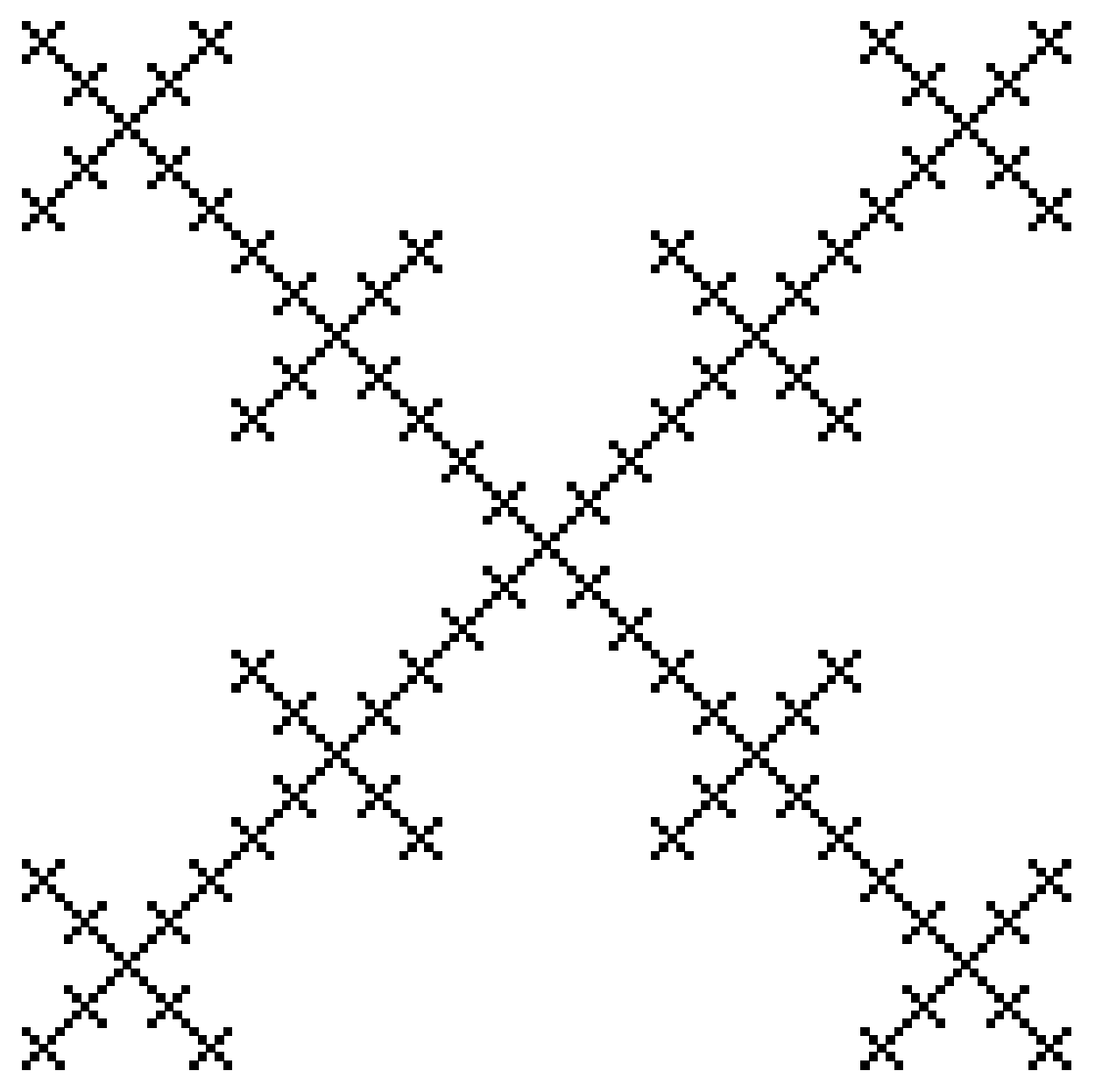}\quad
	\includegraphics[width=4.5cm]{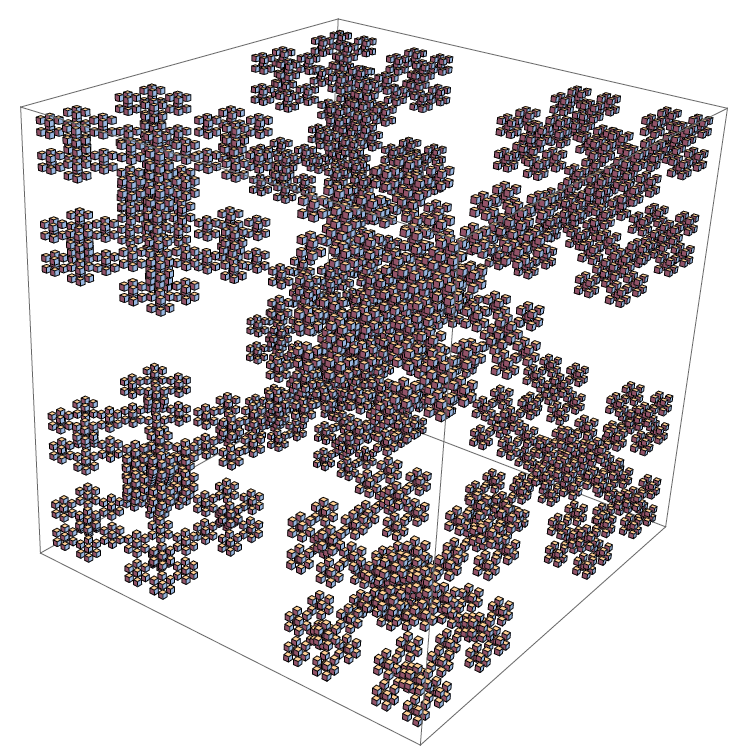}
	\caption{The Vicsek sets $\mathcal{V}_2^2$, $\mathcal{V}_3^2$ and $\mathcal{V}_2^3$.}\label{fig1}
\end{figure}

The simple structure of Vicsek sets leads to many notable results. In particular, in \cite{Z1}, the eigenvalues of the Laplacian on $2$-dimensional Vicsek sets are studied in detail; also, see \cite{BCDEHKMST1} for an earlier work doing research on the fractal $3-$tree. In the first part of our paper, we will follow their idea, and extend some of their results to higher dimensional settings. The method, the spectral decimation recipe, has been a routine argument to solve eigenvalue problems of the Laplacian and was introduced by Shima and Fukushima. See \cite{FS} and \cite{tS} for the celebrated work. For the precise definition of the Laplacian, readers can refer to the books \cite{B,Ki3,S4}, and also read the papers \cite{Ki1,Ki2,L}. One interesting fact is that we can define infinitely many different Dirichlet forms on a Vicsek set even if the renormalization factors are fixed \cite{M}. Also see \cite{HM} for a more general class of tree like Vicsek sets. Finally, readers can find more results about eigenvalues problems in \cite{BCDEHKMST1,BCDEHKMST2,BK,BS,CT,DRS,FKS,HSTZ,MT,ORS,S1,S3,ST,T,Z1,Z2}, including the results of eigenvalue counting functions in \cite{KL}.  

In the second part of the paper, we study the infinite Vicsek set lattices, which are natural generalizations of the Sierpinski lattices.  In \cite{T}, A. Teplyaev showed that two Sierpinski lattices are isomorphic if and only if the generating sequences have the same tail (up to a permutation of $1,2,3$). The result is a consequence of the topological rigidity of the Sierpinski gaskets. Readers also see \cite{BP,S0,S2,ST,T} for results relating to the blow up of the Sierpinski gaskets and other related lattices. On the other hand, there are infinitely many different self-isometries on the Vicsek sets. It is of interest to see when two Vicsek set lattices are equivalent up to isomorphism. The problem will be harder than the Sierpinski gasket case, and the critical observation is that the paths between the center of different level of finite approximation of the lattice determine the structure of the graph.

A brief outline of the paper is as follows. In Section 2, we will introduce notations, including the definition of Vicsek sets $\mathcal{V}_n^d$. In Section 3 and 4, we will deal with the eigenvalues of the discrete Laplacians. In particular, in Section 3 we show the spectral decimation recipe and in Section 4 we count the Dirichlet and Neumann eigenvalues. In Section 5, we study the Vicsek set lattices.

\section{The Vicsek set $\mcV_n^d$.}
In this section, we introduce the class of Vicsek sets $\mcV_n^d$, where $n\geq 2, d\geq 2$, and briefly review the defintion of Laplacians.

To begin with, we fix a $d$-dimensional cube $[0,1]^d\subset \mathbb{R}^d$, and let $\{q_i\}_{i=1}^{2^d}$ be the boundary vertices of the cube, i.e. $\{q_i\}_{i=1}^{2^d}=\{0,1\}^d$.  An ordering of the vertices is not very important, since the Vicsek set has the full $S_{2^d}$ symmetry, but for concreteness, we set the following rule $i\in\{1,2,\cdots, 2^d\}\to \{0,1\}^d$:
\[q_i=(a_1,a_2,a_3,\cdots,a_d)\in \{0,1\}^d,\quad \text{ if and only if }i=1+\sum_{l=1}^d 2^{l-1}a_l.\]
Let $q_0=(\frac12,\frac12,\cdots)$ be the center of the cube, we define the iterated function system (i.f.s.) as follows,
\[F_{i,j}(x)=\frac{1}{2n-1}x+\frac{2n-2}{2n-1}\big(\frac {n-j}{n-1} q_i+\frac {j-1}{n-1} q_0\big),\quad \text{ for }1\leq i\leq 2^d,1\leq j\leq n.\]
The Vicsek set $\mcV_n^d$ is the unique compact set in $\mathbb{R}^d$ such that 
\[\mcV_n^d=\bigcup_{i=1}^{2^d}\bigcup_{j=1}^n F_{i,j}(\mcV_n^d).\]
We will use the notations $p_{i,j}=F_{i,j}(q_i)$ for $1\leq i\leq 2^d,1\leq j\leq n$. In particular, we have 
\[p_{i,1}=F_{i,1}(q_i)=q_i,\quad\forall 1\leq i\leq 2^d.\]
See Figure \ref{fig2} for an illustration of our notations.

\begin{remark}
Actually we only have $2^dn-2^d+1$ contractions in the i.f.s. $\{F_{i,j}:1\leq i\leq 2^d,1\leq j\leq n\}$, since $F_{i,n}=F_{i',n}$ for any $i,i'\in \{1,2,\cdots,2^d\}$. 
\end{remark}

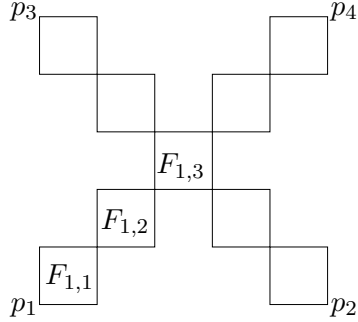
\begin{figure}[htp]
	\includegraphics[width=4cm]{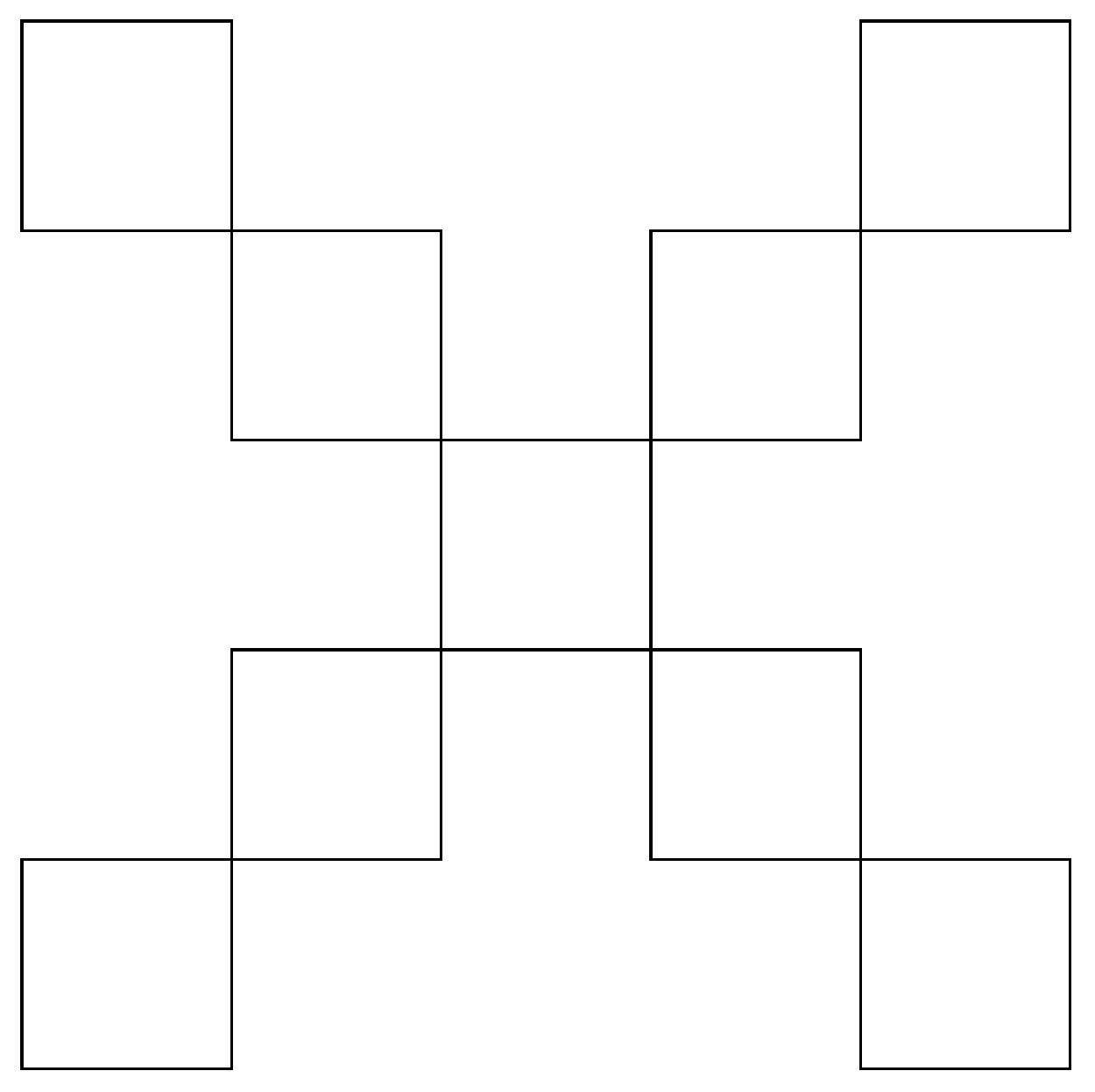}
	\begin{picture}(0,0)
		\put(-126,0){$p_1$}
		\put(-5,0){$p_2$}
		\put(-126,112){$p_3$}
		\put(-5,112){$p_4$}
		
		\put(-113,10){$F_{1,1}$}
		\put(-92,31){$F_{1,2}$}
		\put(-71,52){$F_{1,3}$}
	\end{picture}
	\caption{An illustration of $F_{i,j}$ and $q_i$.}\label{fig2}
\end{figure}

The set $V_0=\{q_i\}_{i=1}^{2^d}$ is treated as the boundary of the Vicsek set, and we let $E_0=\big\{\{q_i,q_{i'}\}:1\leq i<i'\leq 2^d\big\}$. Then $G_0=(V_0,E_0)$ is a complete graph. We define the level-$m$ approximating graph $G_m=(V_m,E_m)$ iteratively as follows:
\[V_m=\bigcup_{i=1}^{2^d}\bigcup_{j=1}^n F_{i,j}V_{m-1},\text{ for }m\geq 1,\]
and $E_m=\big\{\{F_{i,j}x,F_{i,j}y\}:\{x,y\}\in E_{m-1}, 1\leq i\leq 2^d, 1\leq j\leq n\big\}$. For $x,y\in V_m$ we will simply write $x\stackrel{m}\sim y$ if $\{x,y\}\in E_m$. When we use the notation $\sum_{x\stackrel{m}\sim y}$, we are taking the summation of a function defined on $E_m$, and each edge is counted once.

With the above notations, we can define the self-similar resistance forms and the Laplacians on Vicsek sets.

\begin{definition}\label{def22}
Consider $\mcV_n^d$, with $d\geq 2,n\geq 2$. Let $\mu=\mu_n^d$ be the normalized Hausdorff measure on $\mcV_n^d$. Let $m\geq 0$

(a). For any $f\in l(V_m)$, we define $\mcE_m(f)=(2n-1)^m\sum_{x\stackrel{m}\sim y} \big(f(x)-f(y)\big)^2$. 

(b). For each vertex $x\in V_m$ and $f\in l(V_m)$, we define 
\[\Delta_m f(x)=-N_x^{-1}\sum_{x\stackrel{m}\sim y}\big(f(y)-f(x)\big)=N_x^{-1}\sum_{x\stackrel{m}\sim y}f(y)-f(x),\]
where $N_x=\sum_{y:x\stackrel{m}\sim y}1$ is the number of neighbouring vertices of $x$ (degree of $x$).\vspace{0.1cm}

\noindent By taking the limit, we get the resistance forms as follows (see the books \cite{B,Ki3,S4}).

(c). For any $f\in C(\mcV_n^d)$, we define 
\[\mcE(f)=\lim\limits_{m\to\infty}\mcE_m(f),\]
and let $\mcF=\{f\in C(\mcV_n^d):\mcE(f)<\infty\}$. \vspace{0.1cm}

\noindent The form $(\mcE,\mcF)$ is then a local regular Dirichlet form on $L^2(\mcV_n^d,\mu)$. The Laplacian is defined with the following weak formula.

(d). We say $f\in dom(\Delta)$ if $f\in \mcF$ and there is $u\in L^2(\mcV_n^d,\mu)$ such that 
\[\mcE(f,g)=-<u,g>,\quad\forall g\in \mcF_0\]
where $<u,g>=\int_{\mcV_n^d} u(x)g(x)\mu(dx)$, $\mcE(f,g)=\frac{1}{2}\big(\mcE(f+g)-\mcE(f)-\mcE(g)\big)$ and $\mcF_0=\{v\in \mcF:v|_{V_0}=0\}$. We write $u=\Delta f$.
\end{definition}

There is well-known pointwise formula of the Laplacian $\Delta f$, when $\Delta f\in C(\mcV_n^d)$:
\[\Delta f(x)=2^d(2^d-1)\lim_{m\to\infty} \big((2n-1)(2^dn-2^d+1)\big)^m\Delta_m f(x),\qquad\forall x\in V_*,\]
where $V_*=\bigcup_{m=0}^\infty V_m$. The above limit is uniform on $V_*$. See books \cite{B,Ki3,S4} for details.\vspace{0.2cm}

Before the end of this section, we briefly introduce the eigenvalue problems we will study in the next section. 

Let $\lambda\in \mathbb{R}$ and $f\in \mcF$, we say $f$ is an eigenfunction (of the Laplacian) with eigenvalue $\lambda$ if equation (\ref{eqn21}) holds for any $g\in \mcF_0=\{u\in \mcF:u|_{V_0}=0\}$. 
\begin{equation}\label{eqn21}
\mcE(f,g)=\lambda<f,g>.
\end{equation}
In other words, $f$ is in the domain of the Laplacian, and 
\[-\Delta f=\lambda f.\] 
In addition, we say $\lambda$ is a \emph{Neumann} eigenvalue and $f$ is a \emph{Neumann} eigenfunction if and only if (\ref{eqn21}) holds for any $g\in \mcF$; we say $\lambda$ is a \emph{Dirichlet} eigenvalue and $f$ is a Dirchlet eigenfunction if and only if $f\in \mcF_0$.  

The eigenvalues and eigenfunctions corresponding to the graph Laplacian can be defined in a same manner. We simply say $f$ is an eigenfunction (of $\Delta_m$) with eigenvalue $\lambda_m$ if  
\begin{equation}\label{eqn22}
-\Delta_m f(x)=\lambda_mf(x),\quad \forall x\in V_m\setminus V_0.
\end{equation}
In addition, we say $\lambda_m$ is a \emph{Neumann} eigenvalue and $f$ is a \emph{Neumann} eigenfunction if and only if (\ref{eqn22}) holds for any $x\in V_0$; we say $\lambda$ is a \emph{Dirichlet} eigenvalue and $f$ is a Dirchlet eigenfunction if and only if $f|_{V_0}=0$.  

It is well known \cite{FS,tS} that all the Neumann and Dirichlet eigenfunctions are generated as normalized limits of Neumann and Dirichlet eigenfunctions of $\Delta_m$. So we will focus on the discrete eigenfunctions in this paper for simplicity.

\section{Spectral decimation}
The eigenfunctions and eigenvalues on Vicsek sets can be computed exactly with the celebrated spectral decimation recipe. In their acclaimed paper, Fukushima and Shima introduced the method to compute the Dirichlet and Neumann eigenfunctions \cite{FS}. The results extended to p.c.f. self-similar sets with strong regular harmonic structures \cite{tS}.

In particular, in \cite{Z1}, D. Zhou provided a full story about the eigenfunctions on planar Vicsek sets, and our result is a natural generalization. We aim to provide a version that is friendly to readers without any knowledge of strongly regular harmonic structures on fractals. The computations will be similar to that in \cite{S4} Chapter 3.

In this section, we will consider eigenfunctions of $\Delta_m,m\geq 0$. We show that there is a polynomial $R_{n,d}$ (depending on the fractal $V_n^{(d)}$) such that if $f$ is an eigenfunction on $V_{m+1},m\geq 0$, i.e.
\begin{equation}\label{eqn31}
	-\Delta_{m+1} f(x)=\lambda f(x),\quad\forall x\in V_{m+1}\setminus V_0,
\end{equation}
and $\lambda\neq \frac{2^d}{2^d-1}$, then 
\begin{equation}\label{eqn32}
	-\Delta_m f(x)=R_{n,d}(\lambda)f(x),\quad \forall x\in V_m\setminus V_0.
\end{equation}
In particular, if (\ref{eqn31}) holds for $x\in V_0$, so does (\ref{eqn32}), so we do not worry about Neumann eigenfunctions. The reverse direction also holds when $\lambda$ is not a forbidden eigenvalue (see explanation in subsection 3.2.): if we have (\ref{eqn32}) holds, then we can get an extension of $f$ so that (\ref{eqn31}) holds. 

Following \cite{Z1}, we use the Chebyshev polynomials to represent the results. In particular, $T_n(\cos\theta)=\cos(n\theta)$ and $U_n(\cos\theta)=\frac{\sin\big((n+1)\theta\big)}{\sin\theta}$ for $\theta\in \mathbb{R}$. 

\begin{definition}\label{def31}
	Let $T_n,U_n$ be the Chebyshev polynomials of the first kind and the second kind, i.e.,  
	\[
	\begin{cases}
		T_n(t)=2t T_{n-1}(t)-T_{n-2}(t),\\
		U_n(t)=2t U_{n-1}(t)-U_{n-2}(t),
	\end{cases}
	\]
	with $T_0(t)=1, T_1(t)=t, U_0(t)=1, U_1(t)=2t$.  
\end{definition}

\subsection{Restriction}
In this section, we consider an easy case: assume we know (\ref{eqn31}), we want to see (\ref{eqn32}).

\begin{theorem}\label{thm32}
	Let $t=-(2^d-1)\lambda+1$ for short. Define the polynomial \[R_{n,d}(\lambda)=1+(\lambda-1)\cdot U_{n-1}^2(t)+\big(t+\lambda+1\big)\cdot U_{n-1}(t)U_{n-2}(t)-t\cdot U_{n-2}^2(t).\]
	If $f$ is an eigenfunction (or Neumann eigenfunction) on $V_{m+1}$ with eigenvalue $\lambda$, then $f$ is an eigenfunction (or Neumann eigenfunction) on $V_m$ with eigenvalue $R_{n,d}(\lambda)$. 
\end{theorem}

For short, we write $F_w=F_{w_1}F_{w_2}\cdots F_{w_m}$ and $K_w=F_w\mcV_n^{(d)}$, where $w=w_1w_2\cdots w_m$ with $w_i\in \{(i,j):1\leq i\leq 2^d,1\leq j\leq n\}$. The set $K_w$ is called an $m$-cell. We will focus on an $m$-cell, $K_w$. 

Let $f$ be an eigenfunction of $\Delta_{m+1}$ on $V_{m+1}=\bigcup_{w'\in W_m} F_{w'}V_1$, with the eigenvalue being $\lambda$. We write 
\[
\begin{cases}
	J_{i,j}=f(F_wp_{i,j}),&\text{ for }1\leq i\leq 2^d,1\leq j\leq n,\\
	L_{i,j,k}=f(F_wF_{i,j}q_k),&\text{ for }1\leq i\leq 2^d,1\leq j\leq n-1, k\in \{1,2,\cdots,2^d\}\setminus \{i,2^d-i+1\}.
\end{cases}
\]
For short, we write $N=2^d-1$, so the number of neighbours of a vertex is either $N$ or $2N$. In addition, we use the same notation
\[t=1-N\lambda=-(2^d-1)\lambda+1\]
as in the statement of Theorem \ref{thm32}.

Then, the eigenvalue equation (\ref{eqn31}) on $F_w(V_1\setminus V_1)$ can now be rewritten as follows,
\begin{eqnarray}
\label{eqn33}	
\lambda L_{i,j,k}=L_{i,j,k}-\frac{1}{N}\big(J_{i,j}+J_{i,j+1}+\sum_{k'\neq i,2^d-i+1,k}L_{i,j,k'}\big),\\
\label{eqn34}
\lambda J_{i,j}=J_{i,j}-\frac{1}{2N}\big(J_{i,j-1}+J_{i,j+1}+\sum_{k\neq i,2^d-i+1}(L_{i,j,k}+L_{i,j-1,k})\big),\\
\label{eqn35}	
\lambda J_{i,n}=J_{i,n}-\frac{1}{2N}\big(J_{i,n-1}+\sum_{k\neq i,2^d-i+1}L_{i,n-1,k}+\sum_{i'\neq i}J_{i',n}\big),
\end{eqnarray}
where in (\ref{eqn33}), $1\leq i\leq 2^d,1\leq j\leq n-1, k\in \{1,2,\cdots,2^d\}\setminus \{i,2^d-i\}$; in (\ref{eqn34}), $1\leq i\leq 2^d,1\leq j\leq n-1$; in (\ref{eqn35}), $1\leq i\leq 2^d$.

For convenience, in this subsection, we assume $F_wp_1$ is a nonjunction point, then
\begin{equation}\label{eqn36}
	\lambda J_{1,1}=J_{1,1}-\frac{1}{N}\big(J_{1,2}+\sum_{k\neq 1,2^d}L_{1,1,k}\big)
\end{equation}

We will use the equations (\ref{eqn33}),(\ref{eqn34}) and (\ref{eqn35}) to find a relation between $\sum_{i\neq 1}J_{i,1}$ and $J_{1,2}+\sum_{k\neq i,2^d-i+1}L_{1,1,k}$. This will provide all the information we need to compare $\Delta_{m+1}f(F_wq_1)$ and $\Delta_m f(F_wq_1)$. We list the computation steps as lemmas for convenience of readers. 

\begin{lemma}\label{lemma33}
 If $\lambda\neq \frac{2^d}{2^d-1}$, then $L_{i,j,k}=\frac{1}{1+t}(J_{i,j}+J_{i,j+1})$, for any  $1\leq i\leq 2^d,1\leq j\leq n-1, k\in \{1,2,\cdots,2^d\}\setminus \{i,2^d-i+1\}$. 	
\end{lemma}
\begin{proof}
	Fix $1\leq i\leq 2^d,1\leq j\leq n-1$. By summing equations (\ref{eqn33}) over $k'\in \{1,2,\cdots,2^d\}\setminus \{i,2^d-i+1\}$, we get 
	\[(\lambda-1+\frac{N-2}{N})(\sum_{k\neq i,2^d-i+1}L_{i,j,k})=-\frac{N-1}{N}(J_{i,j}+J_{i,j+1}),\]
	so 
	\[\sum_{k\neq i,2^d-i+1}L_{i,j,k}=-\frac{N-1}{N\lambda-2}(J_{i,j}+J_{i,j+1})\]
	For a fixed $k\neq i,2^d-i+1$, we insert the above relation into equation (\ref{eqn33}) again. If $\lambda=\frac{2^d}{2^d-1}$, (\ref{eqn33}) always holds, so we have extra freedom for the value of $L_{i,j,k}$; if $\lambda\neq \frac{2^d}{2^d-1}$, we can see that $L_{i,j,k}=-\frac{1}{N\lambda-2}(J_{i,j}+J_{i,j+1})$.
\end{proof}

Next, we plug the formula of $L_{i,j,k}$ in Lemma \ref{lemma33} into (\ref{eqn33}) and (\ref{eqn36}). 

\begin{lemma}\label{lemma34}
	For $1\leq i\leq 2^d,2\leq j\leq n$, we have 
	\begin{equation}\label{eqn37}
		J_{i,j}=2tJ_{i,j-1}-J_{i,j-2}.
	\end{equation}
    In addition, if $F_wq_1$ is nonjunction, we have 
    \begin{equation}\label{eqn38}
    	 J_{1,2}=tJ_{1,1}. 	
    \end{equation}
\end{lemma}

We can apply (\ref{eqn37}) to find a recursive formula of $J_{i,n}$ in terms of $J_{i,1}$ and $J_{i,2}$. Moreover, we can also do the other direction, finding the formula of $J_{i,1}$ in terms of $J_{i,n}$ and $J_{i,n+1}$, where $J_{i,n+1}$ is a make up term to simplify the computations.

\begin{lemma}\label{lemma35}
For $j\geq 0$, we define $P_j(x)=2T_j(x)-U_j(x)$; for $j\geq 1$, we define $Q_j(x)=U_{j-1}(x)$. In addition, we introduce a ghost term
\[J_{i,n+1}=2tJ_{i,n}-J_{i,n-1}.\]
	
(a). For $1\leq i\leq 2^d$ and $1\leq j\leq n+1$, we have 
\begin{equation}\label{eqn39}
J_{i,j}=P_{j-1}J_{i,1}+Q_{j-1}J_{i,1}.
\end{equation}

(b). For $1\leq i\leq 2^d$ and $1\leq j\leq n+1$, we have
\begin{equation}\label{eqn310}
	J_{i,j}=P_{n+1-j}J_{i,n+1}+Q_{n+1-j}J_{i,n}.
\end{equation}
\end{lemma}

In particular, if $F_wq_1$ is a nonjunction vertex as in our setting, by (\ref{eqn38}) and (\ref{eqn39}), we have 
\begin{equation}\label{eqn311}
	J_{1,j}=T_{j-1}(t)\cdot J_{1,1},\text{ for }1\leq j\leq n. 
\end{equation}

\begin{lemma}\label{lemma36}
We have the equation
	\begin{equation}\label{eqn312}
		\sum_{i'\neq i} J_{i',n}=\frac{t+N}{t+1}J_{i,n+1}+\frac{N-1}{t+1}J_{i,n}.
	\end{equation}
\end{lemma}
\begin{proof}
	Fix $i$, and insert the formula $L_{i,n-1,k}=-\frac{1}{N\lambda-2}(J_{i,n-1}+J_{i,n})$ into (\ref{eqn35}), we get 
	\[
	\frac{1}{2N}\sum_{i'\neq i}J_{i',n}=(1-\lambda)J_{i,n}-\frac{1}{2N}J_{i,n-1}+\frac{N-1}{2N(N\lambda-2)}(J_{i,n}+J_{i,n-1}).
	\]
	By simplifying the above equation, we get
	\[
	\sum_{i'\neq i} J_{i',n}=-\frac{-2N^2\lambda^2+(2N^2+4N)\lambda-3N-1}{N\lambda-2}J_{i,n}+\frac{N+1-N\lambda}{N\lambda-2}J_{i,n-1}.
	\]
	Finally, one can see it is equivalent to (\ref{eqn312}).
\end{proof}

\noindent\textbf{Remark.} The intuition behind (\ref{eqn312}) is that if we let $L_{i,n,k}=\frac{1}{t+1}(J_{i,n}+J_{i,n+1})$, then we have $\sum\limits_{k\neq i,2^d-i+1}L_{i,n,k}+J_{i,n}=\frac{t+N}{t+1}J_{i,n+1}+\frac{N-1}{t+1}J_{i,n}$. \vspace{0.15cm}

In particular, as a consequence of Lemma \ref{lemma36}, we have 
\begin{equation}\label{eqn313}
\sum_{i\neq 1} J_{i,n}=\frac{t+N}{t+1}J_{1,n+1}+\frac{N-1}{t+1}J_{1,n}.
\end{equation}
In addition, by summing (\ref{eqn312}) over $i\neq 1$ cases, we get
\[NJ_{1,n}+(N-1)\sum_{i\neq 1}J_{i,n}=\sum_{i'\neq i}\sum_{i\neq 1}J_{i',n}=\frac{t+N}{t+1}\sum_{i\neq 1}J_{i,n+1}+\frac{N-1}{t+1}\sum_{i\neq 1}J_{i,n},\]
which is simplified to be
\[\sum_{i\neq 1}J_{i,n+1}=\frac{N(t+1)}{t+N}J_{1,n}+\frac{(N-1)t}{t+N}\sum_{i\neq 1}J_{i,n}.\]
Insert equation (\ref{eqn313}) into the above equation, we get
\begin{equation}\label{eqn314}
\sum_{i\neq 1}J_{i,n+1}=\frac{Nt-t}{t+1}J_{1,n+1}+\frac{Nt+1}{t+1}J_{1,n}.
\end{equation}

Finally, by Lemma \ref{lemma35} (b) and equations (\ref{eqn313}) (\ref{eqn314}), we have 
\[
\begin{aligned}
\sum_{i\neq 1} J_{i,1}=&P_n(t)\big(\frac{Nt-t}{t+1}T_n(t)+\frac{Nt+1}{t+1}T_{n-1}(t)\big)\cdot J_{1,1}
+Q_n(t)\big(\frac{N+t}{t+1}T_n(t)+\frac{N-1}{t+1}T_{n-1}(t)\big)\cdot J_{1,1}\\
=&-U_{n-2}(t)\big(\frac{Nt-t}{t+1}T_n(t)+\frac{Nt+1}{t+1}T_{n-1}(t)\big)\cdot J_{1,1}\\&+U_{n-1}(t)\big(\frac{N+t}{t+1}T_n(t)+\frac{N-1}{t+1}T_{n-1}(t)\big)\cdot J_{1,1}\\=&-U_{n-2}(t)\big(\frac{Nt-t}{t+1}(tU_{n-1}(t)-U_{n-2}(t))+\frac{Nt+1}{t+1}(U_{n-1}(t)-tU_{n-2}(t))\big)\cdot J_{1,1}\\&+U_{n-1}(t)\big(\frac{N+t}{t+1}(tU_{n-1}(t)-U_{n-2}(t))+\frac{N-1}{t+1}(U_{n-1}(t)-tU_{n-2}(t))\big)\cdot J_{1,1}\\
\end{aligned}
\]
The eigenvalue equation on $V_{m-1}$ holds at $F_wq1$ with eigenvalue being $\lambda_m$ if 
$$\sum_{i\neq 1} J_{i,1}=N(1-\lambda_m)\cdot J_{1,1}.$$
So
\[
\begin{aligned}
N(1-\lambda_m)=&-U_{n-2}(t)\big(\frac{Nt-t}{t+1}(tU_{n-1}(t)-U_{n-2}(t))+\frac{Nt+1}{t+1}(U_{n-1}(t)-tU_{n-2}(t))\big)\\
&+U_{n-1}(t)\big(\frac{N+t}{t+1}(tU_{n-1}(t)-U_{n-2}(t))+\frac{N-1}{t+1}(U_{n-1}(t)-tU_{n-2}(t))\big)\\
\end{aligned}
\]
Which simplifies to:
\[
\begin{aligned}
	\lambda_m=1+(\lambda-1)\cdot U_{n-1}^2(t)+(t+\lambda+1)\cdot U_{n-1}(t)U_{n-2}(t)-t\cdot U_{n-2}^2(t)
\end{aligned}
\]
Where $t=1-N\lambda$ and $N=2^d-1$. 

The above computation applies to any nonjunction point in $V_{m-1}$. For junction points, the same idea works by taking the summation. This finishes the proof of Theorem \ref{thm32}. 

\subsection{Forbidden eigenvalues}
In this part, we need to consider the reverse direction. To extend an eigenfunction on $V_{m}$ to an eigenfunction on $V_{m+1}$. Still, it suffices to study a $m$ cell, so we take the same notations as the previous part. 

The question is now, if we have $J_{i,1},1\leq i\leq 2^d$ given at first, can we always find suitable values for all $J_{i,j}$ and $L_{i,j,k}$. Since (\ref{eqn31}), or equivalently the equations (\ref{eqn33},\ref{eqn34},\ref{eqn35}), provide us a linear system with $\#(V_1\setminus V_0)$ equations (depending on $\lambda$) and exactly $\#(V_1\setminus V_0)$ variables to solve, we only need to see when the system is degenerate. 

In other words, we will find all the $\lambda$ such that (\ref{eqn33},\ref{eqn34},\ref{eqn35}) have non-trivial solutions, if we are given the trivial boundary values $J_{i,1}=0,\forall 1\leq i\leq 2^d$. Such eigenvalues are called forbidden eigenvalues. In particular, we have observed $\frac{2^d}{2^d-1}$ is one of the forbidden eigenvalues. 

To find the rest, we take the advantage of the symmetry, noticing that any function can be decomposed into the symmetric part and the antisymmetric part. \vspace{0.15cm}

\noindent\textbf{The symmetric case.} \textit{Assume $J_{i,1}=0$ and $J_{i,2}=1$ for all $1\leq i\leq 2^d$.}

For such a solution to exist, we only need to check (\ref{eqn312}). By symmetry, $J_{1,n}=J_{i,n},1\leq i\leq 2^d$, so we have
\[NJ_{1,n}=\frac{t+N}{t+1}J_{1,n+1}+\frac{N-1}{t+1}J_{1,n}.\]
In addition, by Lemma \ref{lemma35} (a), we have $J_{1,n+1}=U_{n-1}(t)$ and $J_{1,n}=U_{n-2}(t)$, so
\[(t+N)U_{n-1}(t)-(tN+1)U_{n-2}(t)=0.\]
Noticing that $tU_{n-1}(t)-U_{n-2}(t)=T_n(t)$ and $U_{n-1}(t)-tU_{n-2}(t)=T_{n-1}(t)$, we can simplify the equation to be 
\[T_n(t)+NT_{n-1}(t)=0.\]
\vspace{0.15cm}

\noindent\textbf{The antisymmetric case.} \textit{Assume $J_{i,1}=0$ for all $1\leq i\leq 2^d$. In addition, $J_{1,1}=1,J_{2,1}=-1$ and $J_{i,1}=0$ for all $3\leq i\leq 2^d$.}

Similar to the symmetric case, by (\ref{eqn312}), we have 
\[-J_{1,n}=\frac{t+N}{t+1}J_{1,n+1}+\frac{N-1}{t+1}J_{1,n}.\]
Thus, such a solution exists if and only if 
\[U_{n-1}(t)+U_{n-2}(t)=0.\]
\vspace{0.15cm}

Combining the above two cases, we finally find all the forbidden eigenvalues.
\begin{theorem}\label{thm37}
Let 
\[\begin{cases}
	S_{n,d}(\lambda)=T_n(t)+(2^d-1)T_{n-1}(t),\\
	A_{n,d}(\lambda)=U_{n-1}(t)+U_{n-2}(t),
\end{cases}\]
where $t=1-(2^d-1)\lambda$. The set of forbidden eigenvalues $\pounds$ is 
\[\pounds=\{\frac{2^d}{2^d-1}\}\bigcup\{\text{roots of }S_{n,d}\}\bigcup\{\text{roots of }A_{n,d}\}.\]

If $\lambda\notin\pounds$ and $f$ is an (Neumann) eigenfunction on $V_m$ with eigenvalue $\lambda_m=R_{n,d}(\lambda)$, then we can extend $f$ uniquely to be an (Neumann) eigenfunction on $V_{m+1}$ with eigenvalue being $\lambda$.
\end{theorem}

The second part of the claim follows easily from Subsection 3.1., by a one to one correspondence between $\sum_{i\neq 1}J_{i,1}$ and $J_{1,2}+\sum_{k\neq 1,2^d}L_{1,1,k}$ when $\lambda\notin\pounds$. \vspace{0.2cm}

Finally, we point out that $\#\pounds=2n$, which means $S_{n,d}$ has $n$ different roots, and $A_{n,d}$ has $n-1$ different roots. This can observed with the property of Chebyshev polynomials, noticing that $T_n(\cos\theta)=n\cos\theta$ and $U_n(\cos\theta)=\frac{\sin\big((n+1)\theta\big)}{\sin\theta}$. Readers can find a proof in Zhou's paper \cite{Z1} Proposition 10, where the arguments essentially work here.

\section{Neumann and Dirichlet eigenfunctions}
For convenience, we fix $n,d$ in this section, and we let 
\[\psi_0,\psi_1,\psi_2,\cdots,\psi_{2n-2}\]
be the inverses of $R_{n,d}$, listed in increasing order. To see that all these branches are well defined to be real functions on $[0,\frac{2^d}{2^d-1}]$, we need to refer to the observation by Zhou \cite{Z1}. 

\begin{lemma}\label{lemma41}
	Let $N=2^d-1$ and $t=1-N\lambda$ as in the last section.
	
	(a). $R_{n,d}(\lambda)=\lambda\cdot A_{n,d}(\lambda)\cdot \big(U_{n-1}(t)+NU_{n-2}(t)\big)$. 
	
	(b). $NR_{n,d}(\lambda)-2^d=S_{n,d}(\lambda)\big(U_{n-2}(t)-U_{n-1}(t)\big)$.
\end{lemma}
\begin{proof}
(a). By direct computation, 
\[
\begin{aligned}
R_{n,d}(\lambda)-\lambda\cdot A_{n,d}(\lambda)\cdot \big(U_{n-1}(t)+NU_{n-2}(t)\big)&=1-U_{n-1}^2(t)+2tU_{n-1}U_{n-2}(t)-U_{n-2}^2(t),\\
&=1-U_{n-1}^2(t)+U_{n-2}(t)U_{n}(t),\\
&=0,
\end{aligned}\]
where the second equality is due to the fact $U_n(t)=2tU_{n-1}(t)-U_{n-2}(t)$ and the last equality is a well-known equality of Chebyshev polynomials.

(b).  We have \\
$$NR_{n,d}(\lambda)-2^d=N+N(\lambda-1)U_{n-1}^2(t)+(Nt+N\lambda+N)U_{n-1}(t)U_{n-2}(t)-NtU_{n-2}^2(t)-2^d$$ \\
Thus,

$\\NR_{n,d}(\lambda)-2^d - S_{n,d}(\lambda)\big(U_{n-2}(t)-U_{n-1}(t)\big) \\$
$= N+N(\lambda-1)U_{n-1}(t)^2+(Nt+N\lambda+N)U_{n-1}(t)U_{n-2}(t)-NtU_{n-2}(t)^2-2^d -N -1 -(tU_{n-1}(t)-U_{n-2}(t)+NU_{n-1}(t)-NtU_{n-2}(t))\big(U_{n-2}(t)-U_{n-1}(t)\big),\\
= -1 + N(\lambda-1)U_{n-1}(t)^2+(Nt+N\lambda+N)U_{n-1}(t)U_{n-2}(t)-NtU_{n-2}(t)^2-2^d -((t +N)U_{n-1}(t)+ (-1-Nt)U_{n-2}(t))\big(U_{n-2}(t)-U_{n-1}(t)\big), \\
= -1 + N(\lambda-1)U_{n-1}^2(t)+(Nt+N\lambda+N)U_{n-1}(t)U_{n-2}(t)-NtU_{n-2}(t)^2-2^d - ((-t-N)U_{n-1}(t)^2+(t+N+1+Nt)U_{n-1}(t)U_{n-2}(t)+(-1-Nt)U_{n-2}(t)^2), \\
= -1 + (N\lambda + t)U_{n-1}^2(t) +(N\lambda-t-1)U_{n-1}(t)U_{n-2}(t)+U_{n-2}^2(t), \\
= -1 + U_{n-1}^2(t)+(-2t)U_{n-1}(t)U_{n-2}(t)+U_{n-2}^2(t), \\
= -(1-U_{n-1}^2(t)+2tU_{n-1}U_{n-2}(t)-U_{n-2}^2(t)),\\
= - (1-U_{n-1}^2(t)+U_{n-2}(t)U_{n}(t)),\\
= 0$ \\
where the second to last equality is due to $U_n(t)=2tU_{n-1}(t)-U_{n-2}(t)$ and the last equality is a well-known equality of Chebyshev polynomials, as in part a.
\end{proof}

In particular, the observation shows that $R_{n,d}(\lambda)=0$ and $R_{n,d}(\lambda)=\frac{2^d}{2^d-1}$ have $2n-1$ different roots in $[0,1]$, by using the properties of Chebyshev polynomials. See \cite{Z1} Proposition 10 for the details. In particular, this implies that $R_{n,d}(\lambda)=\lambda'$ has $2n-1$ different roots in $[0,1]$ provided that $\lambda'\in [0,\frac{2^d}{2^d-1}]$.

With the above discussions, we now can define compositions of $\psi_l,0\leq l\leq 2n-2$. For short, let $v\in \{0,1,2,\cdots,2n-2\}^m$, we define
\[\psi_v=\psi_{v_m}\circ \psi_{v_{m-1}}\circ \cdots \psi_{v_2}\circ \psi_{v_1}\]
We reverse the order since each time we apply an additional $\psi_m$ on the left, which represents the decimation of eigenvalues. In particular, we set $\psi_\emptyset$ to be the identity map. 

\subsection{Neumann eigenfunctions}
We briefly talk about Neumann eigenfunctions. The following result is the same as Theorem 14 of \cite{Z1}. 	For convenience, we write 
\[\Lambda_m=\bigcup_{m'=0}^m\{0,1,2,3,\cdots,2n-2\}^{m'}.\]
For each $v=v_1v_2\cdots v_l\in \Lambda_m$ (clearly $l\leq m$), we write $l=|v|$. In particular, we write $0=|\emptyset|$.

\begin{theorem}\label{thm42}
	The set of Neumann eigenfunctions is 
	\[\begin{aligned}
	\sigma_{m,N}=\big\{\psi_v(\frac{2^d}{2^d-1}):&v\in \Lambda_m, v=\emptyset\text{ or }v_1=1,3,5,\cdots, 2n-3\big\}\\
	&\bigcup\big\{\psi_v(0):v\in \Lambda_m, v=\emptyset\text{ or }v_1=2,4,6,\cdots, 2n-2\big\}.
	\end{aligned}\]
	The multiplicities of the Neumann eigenvalues are as follows
	\[\begin{cases}
		M_{m,N}(\psi_v(\frac{2^d}{2^d-1}))=(2^dn-2^d+1)^{m-|v|}(2^d-2)+1,&\text{ with  }v=\emptyset\text{ or }v_1=1,3,5,\cdots 2n-3,\\
		M_{m,N}(\psi_v(0))=1, &\text{ with  }v=\emptyset\text{ or }v_1=2,4,6,\cdots 2n-2.
	\end{cases}
	\]
\end{theorem}
\begin{proof}
The proof is done by a standard counting argument, which is essentially the same as \cite{Z1}. First, by a same proof as Proposition 10 equation (4.2) of \cite{Z1}, one can see that 
\[\psi_l(\frac{2^d}{2^d-1})\notin\pounds,\qquad\text{ if }l=1,3,5,\cdots, 2n-3,\]
and by Lemma \ref{lemma41}, we have $\psi_l(\lambda)\notin \pounds$ for any $\lambda\in (0,\frac{2^d}{2^d-1})$ and $0\leq l\leq 2n-2$. So the decimation recipe works perfectly to see
\[\big\{\psi_v(\frac{2^d}{2^d-1}):v\in \Lambda_m, v=\emptyset\text{ or }v_1=1,3,5,\cdots, 2n-3\big\}\subset \sigma_{m,N},\]
and similarly one can check that 
\[\big\{\psi_v(0):v\in \Lambda_m, v=\emptyset\text{ or }v_1=2,4,6,\cdots, 2n-2\big\}\subset \sigma_{m,N}.\]

Next, we consider the multiplicity of eigenfunctions. Clearly, the descendants of $0$ have multiplicity $1$ each. We consider eigenvalues generated from $\frac{2^d}{2^d-1}$ in the following, $f\in l(V_m)$ is a Neumann eigenfunction of $\Delta_m$ with eigenvalue $\frac{2^d}{2^d-1}$, if and only if  $\sum_{i=1}^{2d}f(F_wq_i)=0$ for any $w\in W_m$. This provides 
\[(2^d n-2^d+1)^m\] 
linearly independent equations, and we have 
\[\# V_m=(2^dn-2^d+1)^m(2^d-1)+1\]
variables, so the mutliplicity of $\frac{2^d}{2^d-1}$ as a Neumann eigenvalue of $-\Delta_m$ is 
\[M_{m,N}(\frac{2^d}{2^d-1})=(2^dn-2^d+1)^m(2^d-1)+1-(2^d n-2^d+1)^m=(2^dn-2^d+1)^m(2^d-2)+1.\]
By using spectral decimation recipe, we can then see the multiplicity of $\psi_v(\frac{2^d}{2^d-1})$ as a Neumann eigenvalue of $-\Delta_m$ is 
\[M_{m,N}\big(\psi_v(\frac{2^d}{2^d-1})\big)=M_{m-|v|,N}\big(\frac{2^d}{2^d-1}\big)=(2^dn-2^d+1)^{m-|v|}(2^d-2)+1.\]
Finally, we need to check that we have all the Neumann eigenvalues, which is verified by the following equality
\[\begin{aligned}
\# V_m=&\big((2^dn-2^d+1)^m(2^d-2)+1\big)\\
&+\sum_{m'=0}^{m-1}(n-1)(2n-1)^{m'}\big((2^dn-2^d+1)^{m-m'-1}(2^d-2)+1\big)\\
&+1+(n-1)\sum_{m'=0}^{m-1} (2n-1)^{m'}.
\end{aligned}\]
\end{proof}

\subsection{Dirichlet eigenfunctions}
Clearly, we will have all the forbidden eigenvalues involved when talking about Dirichlet eigenfunctions, compared with Neumann cases, where only descendents of $\frac{2^d}{2^d-1}$ and $0$ are involved. 

Following Zhou \cite{Z1}, we name roots of $S_{n,d}$ as $\alpha_1,\alpha_2,\cdots, \alpha_n$, and name roots of $A_{n,d}$ as $\beta_1,\beta_2,\cdots,\beta_{n-1}$. 

\begin{theorem}\label{thm43}
	The set of Dirichlet eigenfunctions is 
	\[
	\begin{aligned}
	\sigma_{m,D}=\big\{\psi_v(\frac{2^d}{2^d-1}):v\in \Lambda_{m-1}, v=\emptyset\text{ or }v_1=1,3,5,&\cdots, 2n-3\big\}\\
	\bigcup\{\psi_v(\alpha_i):1\leq i\leq n&,v\in \{0,1,2,\cdots,2n-2\}^{m-1}\}\\
	&\bigcup \{\psi_v(\beta_l):1\leq i\leq n-1,v\in \Lambda_{m-1}\}. 
    \end{aligned}
    \]
	The multiplicities of the Dirichlet eigenvalues are as follows
	\[
	\begin{cases}
		M_{m,D}(\psi_v(\frac{2^d}{2^d-1})=(2^dn-2^d+1)^{m-|v|}(2^d-2)-2^n+1,&\text{ with  }v=\emptyset\text{ or }v_1=1,3,5,\cdots 2n-3,\\
		M_{m,D}(\psi_v(\alpha_i))=1, &\text{ with  }1\leq i\leq n,v\in \{0,1,2,\cdots, 2n-2\}^{m-1},\\
		M_{m,D}(\psi_v(\beta_i))=2^d-1, &\text{ with  }1\leq i\leq n-1,v\in \Lambda_{m-1}.
	\end{cases}
	\]
\end{theorem}
\begin{proof}
The multiplicity of $\psi_v(\frac{2^d}{2^d-1})$ as a Dirichlet eigenvalue of $\Delta_m$ is $(2^dn-2^d+1)^{m-|v|}(2^d-2)+1-2^d$ for a similar reason as the Neumann case. We have the multiplicity $2^d$ fewer than the Neuamnn cases as we have the boundary values are fixed to be $0$. 

For $\Delta_k$, $k\geq 1$, one can construct $2^d-1$ linearly independent Dirichlet eigenfunctions of $\beta_i$, $1\leq i\leq n-1$, by chaining the antisymmetric eigenfunctions of $\Delta_1$. For $\Delta_1$, one has one Dirichlet eigenfunction of $\alpha_i$, $1\leq i\leq n$. So we have the lower bound of the multiplicity as the same numbers in the statement of the theorem.

Finally, one can use a counting argument to see that we have found all the eigenvalues with the correct multiplicities. 
\end{proof}

\section{The Vicsek set lattices}
Lastly, we have a brief study on the equivalence of infinite Vicsek set lattices. This section is inspired by \cite{T} by A. Teplyaev. The notations are almost the same as Section 2-4, except that we use $0$ to replace the pairs $(i,n)$, $1\leq i\leq 2^d$.

\begin{notation}
	(a). Let $W_0=\{\emptyset\}$, $W_1=\{(i,j):1\leq i\leq 2^d,1\leq j\leq n-1\}\bigcup\{0\}$ and $W_m=W_1^m$ for $m\geq 1$. Write $W_*=\bigcup_{m\geq 0}W_m$, which is the set of finite words. 
	
	(b). Denote $\Sigma=W_1^{\mathbb{N}}$. For $m\geq 1$ and $\omega\in \Sigma$, we write $[\omega]_m=\omega_1\omega_2\cdots\omega_m$; for $m=0$ and $\omega\in \Sigma$, we write $[\omega]_0=\emptyset$. Clearly, $[\omega]_m\in W_m$ for $m\geq 0$.  
	
	(c). For each $w=w_1w_2w_3\cdots w_m\in W_m$, we define \[F_w=F_{w_1}F_{w_2}F_{w_3}\cdots F_{w_m}\text{, and } \phi_w=F_{w_1}^{-1}F_{w_2}^{-1}\cdots F_{w_m}^{-1},\]
	where we set $F_0=F_{i,n}$ for $1\leq i\leq 2^d$. For convenience, we write $F_\emptyset=\phi_\emptyset=Id$, where $Id$ is the identity map.
	
	Notice that $\phi_w=\big(F_{w_m}F_{w_{m-1}}\cdots F_{w_1}\big)^{-1}$, so we do not have $F_w^{-1}=\phi_w$ in general.
	
	(d). Write $V_0=\{q_i\}_{i=1}^{2^d}$ and $\tilde{V}_0=V_0\cup\{q_0\}$, where $q_0=\frac{1}{2^d}\sum_{i=1}^{2^d}q_i$. In addition, we define 
	\[V_m=\bigcup_{w\in W_m}F_wV_0,\quad \tilde{V}_m=\bigcup_{w\in W_m}F_w\tilde{V}_0. \]
	In particular, the definition of $V_m,m\geq 1$ is the same as in Section 2-4.
\end{notation}

We define the Vicsek lattices as below, which depends on $d,n$ and an infinite sequence $\omega\in \Sigma$. For short, we always fix $d,n$ in this section.

\begin{definition}\label{def51} Let $\omega\in \Sigma$.
	
(a). Let $V_{-m}=\phi_{[\omega]_m}V_m$. The Vicsek set lattice corresponding to $\omega$ is defined as $V=\bigcup_{m\geq 0}V_{-m}$. 

(b). Let $\tilde{V}_{-m}=\phi_{[\omega]_m}\tilde{V}_m$. The Vicsek set tree corresponding to $\omega$ is defined as $\tilde{V}=\bigcup_{m\geq 0}\tilde{V}_{-m}$.  
\end{definition}

Notice that there are natural graphs defined on the lattices.

\begin{definition}\label{def52}
Let $\omega\in \Sigma$ and $V,\tilde{V}$ be the corresponding lattices. 
	
(a). Define $E=\big\{\{x,y\}:\text{There exists }m\geq 0 \text{ and }w\in W_m\text{ such that }\{x,y\}\subset \phi_{[\omega]_m}F_wV_{0}\big\}$. 

(b). Define $\tilde{E}=\big\{\{x,y\}:\text{There exists }m\geq 0 \text{ and }w\in W_m\text{ such that }x\in \phi_{[\omega]_m}F_wV_{0}\text{, and }y= \phi_{[\omega]_m}F_wq_0\big\}$.

\noindent Then, $G=(V,E)$ and $\tilde{G}=(\tilde{V},\tilde{E})$ are two infinite graphs. 
\end{definition}

Given two Vicsek set lattices $V$ and $V'$, we say that they are isomorphic if the corresponding graphs $G$ and $G'$ are isomorphic. We define the isomorphism between $\tilde{V}$ and $\tilde{V}'$ in a same manner. It is easy to see that $V$ is isomorphic to $V'$ if and only if $\tilde{V}$ is isomorphic to $\tilde{V}'$.

\begin{lemma}\label{lemma53}
Let $V$ be the Vicsek set lattice generated by $\omega$, and $\tilde{V}$ be the related Vicsek set tree. Also, let $V'$ be the related Vicsek set lattice generated by $\omega'$, and $\tilde{V}'$ be the related Vicsek set tree. Then $V$ is isomorphic to $V'$ if and only if $\tilde{V}$ is isomorphic to $\tilde{V}'$.
\end{lemma}
 
In the following, we study the isomorphism between Vicsek set lattices. Since the Vicsek set has weaker connectivity than the Sierpinski gasket, the condition that two lattices are isomorphic is looser here. 

As a first step, we simplify the problem to studying a tree defined with $\phi_{[\omega]_m}q_0$. A path between two vertices $x,y$ in a graph is a finite sequence of vertices $x=x_0,x_1,x_2,\cdots,x_l=y$ such that $\{x_j,x_{j+1}\}\in \tilde{E}$. We call the path self-avoiding if $x_i\neq x_j$ for $i\neq j$.  

\begin{definition}\label{def54}
Let $\omega\in \Sigma$ and $\tilde{V}$ be the corresponding Vicsek set tree. 

(a). Since $\tilde{V}$ is an infinite tree, for any $x,y\in \tilde{V}$, we have a unique self-avoiding path $\gamma_{x,y}$ between $x,y$ in $\tilde{V}$.

(b). Write 
\[\Gamma_{\omega,m}=\bigcup_{k,l\geq m} \gamma_{\phi_{[\omega]_k}q_0,\phi_{[\omega]_l}q_0},\]
which is the union of paths between center points of $\tilde{V}_{-k},k\geq m$.  
\end{definition}

In particular, $\Gamma_{\omega,m}$ is a subtree of $\tilde{V}$ with the edges induced from $\tilde{G}$.

\begin{proposition}\label{prop55}
Let $\omega,\omega'\in \Sigma$ and $V,V'$ be the corresponding Vicsek set lattices, then the following are equivalent.
	
(a). There exists $m\geq 1$ and an isomorphism $\Psi:\Gamma_{\omega,m}\to \Gamma_{\omega',m}$ such that $\Psi(\phi_{[\omega]_k}q_0)=\Psi(\phi_{[\omega']_k}q_0),\forall k\geq m$.

(b). $V$ is isomorphic to $V'$.  

(c). $\tilde{V}$ is isomorphic to $\tilde{V}'$. 
\end{proposition}

\begin{proof}
(b),(c) are equivalent by Lemma \ref{lemma53}.	We will show that (a) and (c) are equivalent.\vspace{0.15cm}

First, we show (c) implies (a). For convenience, we call $A=\phi_{[\omega]_{m+k}}F_\tau \tilde{V}_m$ a $(-m)$ cell of $\tilde{V}$ if $m,k\geq 0$ and $\tau\in W_k$. Clearly, $A$ is then a copy of $\tilde{V}_{-m}$. The same can be defined for $\tilde{V}'$. Let $\Phi:\tilde{V}\to \tilde{V}'$ be an isomorphism. We first show that $\Phi(\tilde{V}_{-m})$ is a $(-m)$ cell of $\tilde{V}'$. Let $\tilde{V}^{(-m)},\tilde{V}'^{(-m)}$ be the sparser trees, i.e.  $\tilde{V}^{(-m)}=\bigcup_{m'=m}^{\infty}\phi_{[\omega]_{m'}}\tilde{V}_{m'-m}$ and $\tilde{V}'^{(-m)}=\bigcup_{m'=m}^{\infty}\phi_{[\omega']_{m'}}\tilde{V}_{m'-m}$.  First, it is clear that $\Phi:\tilde{V}^{(-1)}\to \tilde{V}'^{(-1)}$, as one can check that the center of $(-1)$ cells are characterized as vertices with $2^d$ neighbours and each neighbour is a degree $2$ vertice. Clearly, we can view $\Phi:\tilde{V}^{(-1)}\to \tilde{V}'^{(-1)}$, and it follows that $\Phi:\tilde{V}^{(-2)}\to \tilde{V}'^{(-2)}$ by a same idea. Repeating the argument, we can see that $\Phi:\tilde{V}^{(-m)}\to \tilde{V}'^{(-m)}$ for any $m\geq 1$. It follows easily that $\Phi$ maps $\tilde{V}_{-m}$ onto a $(-m)$ cell of $\tilde{V}'$.

Since $\tilde{V}'=\bigcup_{m=0}^\infty \Phi(\tilde{V}_{-m})$, there exists $m_0\geq 0$ such that $\tilde{V}'_{0}\in \Phi(\tilde{V}_{-m_0})$, which implies $\Phi(\tilde{V}_{-m})=\tilde{V}_{-m}'$ for any $m\geq m_0$. As a consequence, 
\[\Phi(\phi_{[\omega]_k}q_0)=\phi_{[\omega']_k}q_0,\quad\forall m\geq m_0.\]
It follows immediately that $\Phi:\gamma_{\phi_{[\omega]_k}q_0,\phi_{[\omega]_l}q_0}\to \gamma_{\phi_{[\omega']_k}q_0,\phi_{[\omega']_l}q_0}$ for any $k\geq m_0$, since the self-avoiding paths are unique. (a) follows immediately. \vspace{0.15cm}

Next, we show (a) implies (c). Clearly, the result holds for the trivial case $\omega$ has a tail $000\cdots$, and we only need to consider $\omega$ with untrivial tails. Let $\Psi:\Gamma_{\omega,m}\to\Gamma_{\omega',m}$ be an isomorphism, we will extend it to $\Phi:\tilde{V}\to \tilde{V}'$. We start by showing that there is an isomorphism $\Phi_m:\tilde{V}_{-m}\to \tilde{V}'_{-m}$.  It is not hard to see that $\Gamma_{\omega,m}\cap \tilde{V}_{-m}$ consists of at most two `arms' of the form, 
\[\gamma_{\phi_{[\omega]_m}q_0,\phi_{[\omega]_m}q_i},\quad 1\leq i\leq 2^d.\]
Also, since $\Psi$ is an isomorphism from $\Gamma_{\omega,m}$ to $\Gamma'_{\omega,m}$, the numbers of the `arms' in $\Gamma_{\omega,m}\cap V_{-m}$ and $\Gamma'_{\omega,m}\cap \tilde{V}_{-m'}$ are the same. Thus, we have an extension (not unique!) $\Phi_m:\tilde{V}_{-m}\to \tilde{V}'_{-m}$ of $\Psi$. Next, we consider $n=m+1$, and show the existence of $\Phi_{m+1}:\tilde{V}_{-m-1}\to \tilde{V}'_{-m-1}$ such that 
\[\Phi_{m+1}|_{V_{-m}}=\Phi_{m},\quad \Phi_{m+1}|_{\Gamma_{\omega,m}\cap V_{-m-1}}=\Psi|_{\Gamma_{\omega,m}\cap V_{-m-1}}.\]
We consider two cases: \vspace{0.15cm}

\noindent\textit{Case 1: $\omega_{m+1}=0$.} In this case, $\phi_{[\omega]_m}q_0=\phi_{[\omega]_{m+1}}q_0$, so $\phi_{[\omega']_m}q_0=\phi_{[\omega']_{m+1}}q_0$ by the isomorphism $\Psi$. We can see that $\Gamma_{\omega,m}\cap \tilde{V}_{-m-1}$ consists of longer `arms' of the form, 
\[\gamma_{\phi_{[\omega]_{m+1}}q_0,\phi_{[\omega]_{m+1}}q_i},\quad 1\leq i\leq 2^d,\]
and the choice of arms ($i$) is the same as for $\Gamma_{\omega,m}\cap \tilde{V}_{-m}$, and we see the same for $\tilde{V}'_{-m}$. So it is clear that there exists a desired extension.\vspace{0.15cm}

\noindent\textit{Case 2: $\omega_{m+1}\neq 0$.} For this case, $\Gamma_{\omega,m+1}\cap \tilde{V}_{-m-1}$ consists of at most two arms of the form $\gamma_{\phi_{[\omega]_{m+1}}q_0,\phi_{[\omega]_{m+1}}q_i}, 1\leq i\leq 2^d$. Then, one can see that either one of the arms goes across the cell $\tilde{V}_{-m}$, or we have $\Gamma_{\omega,m}=\gamma_{\psi_{[\omega]_m}q_0,\psi_{[\omega]_{m+1}}q_0}\bigcup \Gamma_{\omega,m+1}$. In either case, a desired extension is feasible. \vspace{0.15cm}

So we get a desired extension for both cases. We can keep the same idea for $m+2,m+3$ and so on (only need to restrict to $\Psi|_{\Gamma_{\omega,m+1}},\Psi|_{\Gamma_{\omega,m+2}},\cdots$), which finally gives us an extension $\Phi$.  
\end{proof}

By proposition \ref{prop55}, we only need to study the subtree $\Gamma_{\omega,m}$.

\begin{theorem}\label{thm56}
Let $\omega,\omega'\in \Sigma$. Let $V$ ($V'$) be the Vicsek lattices associated with $\omega$ ($\omega'$). In addition, if $\omega_m\neq 0$ ($\omega'_m\neq 0$), then we write $\omega_m=(i_m,j_m)$ ($\omega'_m=(i'_m,j'_m)$).

$V$ is isomorphic to $V'$ if and only if there is $M>0$ such that (a),(b) hold for all $m\geq M$.

(a). $\omega_m=0$ if and only if $\omega'_m=0$.

(b). If $\omega_m\neq 0$, then $j_m=j'_m$.

(c). If $\omega_m\neq 0$ and $M\leq m'\leq m$, then 
\[\omega_{m''}=0\text{ or } i_{m''}=2^d+1-i_m, \quad \forall m'\leq m''<m,\]
if and only if
\[\omega'_{m''}=0\text{ or } i'_{m''}=2^d+1-i'_m, \quad \forall m'\leq m''<m.\]
\noindent 
\end{theorem}

\begin{proof}
The theorem is easy to see as long as we understand the geometric meaning of the conditions in (a), (b) and (c):
	
\noindent In (a), $\omega_m=0$ is equivalent to the claim that $\gamma_{\phi_{[\omega]_{m-1}}q_0,\phi_{[\omega]_m}q_0}$ is empty. 

\noindent In (b), $j_m$ determines the length of $\gamma_{\phi_{[\omega]_{m-1}}q_0,\phi_{[\omega]_m}q_0}$.
 
\noindent In (c), $\omega_{m''}=0$ or $i_{m''}=2^d+1-i_m$ for any $m'\leq m''<m$ is equivalent to the claim that $\gamma_{\phi_{[\omega]_{m''-1}}q_0,\phi_{[\omega]_{m''}}q_0}$ is a subpath of $\gamma_{\phi_{[\omega]_{m'-1}}q_0,\phi_{[\omega]_{m'}}q_0}$ (in the reverse direction). 

By Proposition \ref{prop55}, if $V$ is isomorphic to $V'$, there is some $M\geq 1$ such that $\Gamma_{\omega,M}$ is isomorphic to $\Gamma_{\omega',M}$. So (a),(b),(c) hold for their geometric meaning. 

Next, assuming (a),(b),(c) hold for $m$ and $\bigcup_{M\leq k,l\leq m-1} \gamma_{\phi_{[\omega]_k}q_0,\phi_{[\omega]_l}q_0}$ is isomorphic to $\bigcup_{M\leq k,l\leq m-1} \gamma_{\phi_{[\omega']_k}q_{0},\phi_{[\omega']_l}q_{0}}$, then we have $\bigcup_{M\leq k,l\leq m} \gamma_{\phi_{[\omega]_k}q_{0},\phi_{[\omega]_l}q_{0}}$ is isomorphic to $\bigcup_{M\leq k,l\leq m} \gamma_{\phi_{[\omega']_k}q_{0},\phi_{[\omega']_l}q_{0}}$, since we can extend the isomorphism to the new edges. So the other direction is proven.
\end{proof}

\bibliographystyle{amsplain}

\end{document}